\documentclass[12pt]{amsart}
\usepackage[a4paper,twoside,left=2cm,right=2cm,top=3.1cm,bottom=2.3cm]{geometry}

\usepackage{amsmath,amssymb,amscd,amsthm}
\usepackage{latexsym}
\usepackage{graphicx}
\usepackage[english]{babel}
\usepackage[latin1]{inputenc}       
\usepackage{textcomp}
\usepackage{times}
\usepackage{colortbl}

\newtheorem{theorem}{Theorem}[section]
\newtheorem{corollary}[theorem]{Corollary}
\newtheorem{proposition}[theorem]{Proposition}
\newtheorem{lemma}[theorem]{Lemma}

\newtheorem{remark}[theorem]{Remark}

\def\irr#1{{\rm Irr}(#1)}
\def\irrr#1#2 {\irr {#1 \mid #2}}

\newcommand{\R}{\mathbb R}

\newcommand{\sfe}{{{\mathbb S}^{n-1}}}

\newcommand{\C}{\mathcal C}

\renewcommand{\d}{\mbox{\rm d}}

\begin{document}

\title[]{
Geometric properties of solutions to elliptic PDE's in Gauss space\\
And related Brunn-Minkowski type inequalities}
\author[Andrea Colesanti, Lei Qin, Paolo Salani ]{Andrea Colesanti, Lei Qin, Paolo Salani}
\address{Dipartimento di Matematica e Informatica ``U. Dini", Universit\`a degli Studi di Firenze}
\email{andrea.colesanti@unifi.it}
\address{Institute of Mathematics,
Hunan University, Changsha, 410082, China}
\email{qlhnumath@hnu.edu.cn}
\address{Dipartimento di Matematica e Informatica ``U. Dini", Universit\`a degli Studi di Firenze}
\email{paolo.salani@unifi.it}
\keywords{Log-concave functions; Convex bodies; Brunn-Minkowski inequality; Viscosity solutions; Gaussian space}

\begin{abstract}
We prove a Brunn-Minkowski type inequality for the first (nontrivial) Dirichlet eigenvalue of the weighted $p$-operator
\[
-\Delta_{p,\gamma}u=-\text{div}(|\nabla u|^{p-2} \nabla u)+(x,\nabla u)|\nabla u|^{p-2},
\]
where $p>1$, in the class of bounded Lipschitz domains in $\R^n$. We also prove that the corresponding positive eigenfunctions are log-concave if the domain is convex.
\end{abstract}

\maketitle
\baselineskip18pt

\parskip3pt
\section{Introduction}

The classical Brunn-Minkowski inequality has a deep impact on both geometry and analysis, see the beautiful survey paper by Gardner \cite{Gardner2002} for more information. Its tentacles extend in many directions and, noticeably, Brunn-Minkowski type inequalities have been proved for several variational functionals, see for instance \cite{Borell1, Borell3, Brascamp-Lieb1976, Colesanti, Colesanti-Salani}.
In particular, the Brunn-Minkowski inequality for the first Dirichlet eigenvalue of the Laplace operator \cite{Brascamp-Lieb1976} has been extended to the cases of the $p$-Laplace operator \cite{Andrea-Cuoghi-Salani2006}, of the Monge-Amp\`ere operator \cite{Salani2005}, of the $k$-Hessian operators \cite{LiuMaXu, Salani2012}, and of the principal frequency of fully non-linear homogeneous elliptic operators \cite{Crasta-Fragal2020}.
Moreover, very recently, a Brunn-Minkowski inequality has been established for the principal frequency of the Laplace operator in the Gauss space, see  \cite{Andrea-Paolo2024}.
Roughly speaking, the main goal of this paper is to extend the latter result to the $p$-Laplace operator in the Gauss space, for $p>1$.

Noticeably, every Brunn-Minkowski type inequality mentioned above is strictly connected to a suitable convexity property of the minimizing functions of the variational functionals at hand. As far as we know, in the known cases of an eigenvalue/principal frequency, this property turns out to be log-concavity. Ultimately, the same happens for the principal frequency of the Gaussian $p$-Laplacian, indeed here we also prove the log-concavity of the associated eigenfunctions in convex domains.


Now, let us describe more precisely the results of this paper. Let $\Omega$ be a bounded domain in $\R^n$. For $p>1$, the $p$-th Rayleigh quotient in Gauss space is given by
\begin{equation}\label{varia-defi}
\frac{\int_{\Omega} |\nabla u(x)|^{p} d\gamma}{\int_\Omega |u(x)|^{p} d\gamma}.
\end{equation}
Here $\gamma$ denotes the Gaussian probability measure in $\R^n$, given by
$$
\gamma(A)= \frac{1}{(2\pi)^{n/2}} \int_{A} e^{-|x|^2/2} dx\,,\quad\text{for any measurable set }A\,.
$$
We consider the following minimum problem:
$$
\lambda_{p,\gamma}:=\inf\left\{
\frac{\int_{\Omega} |\nabla u|^{p} d\gamma}{\int_\Omega |u|^{p} d\gamma}\colon u\in W^{1,p}_0(\Omega,\gamma),\, \int_\Omega |u|^{p} d\gamma>0
\right\}.
$$
The corresponding Euler-Lagrange equation is given by
\begin{equation}\label{PDE1}
\left\{
\begin{aligned}
& -\Delta_{p,\gamma} u=\lambda_{p,\gamma} |u|^{p-2}u \ \  &\mathrm {in}\ \ \Omega, \\
&\ u=0 \ \ \ \ &\mathrm {on} \ \ \partial \Omega\,,
\end{aligned}
\right.
\end{equation}
where
$$
-\Delta_{p,\gamma} u:=-\text{div}(|\nabla u|^{p-2} \nabla u)+(x,\nabla u)|\nabla u|^{p-2}\,.
$$
Any (weak) solution of problem \eqref{PDE1} is called a first (Dirichlet) eigenfunction. For the existence, we can refer, for instance, to \cite[section 6]{Franceschi2024}.
By standard arguments (based on the ones of \cite{Lindqvist1990}), it follows that all first eigenfunctions coincide up to scalar multiplication, and they don't change sign inside $\Omega$, so we will consider only nonnegative eigenfunctions, if not differently specified.
More details can be found in \cite{Feng2021} and \cite{Franceschi2024}, where the authors give basic introductory material for problem \eqref{PDE1}.

\medskip

As already said, one of our main results is a Brunn-Minkowski type inequality for $\lambda_{p,\gamma}$, in the class of open bounded domains with Lipschitz boundary. {   To avoid useless complications, throughout the paper we will consider only connected sets, even if not explicitly specified.}
\begin{theorem}\label{T1}
Fix $p>1$ and $t\in[0,1]$. Let $\Omega_0,\Omega_1$  and
\[
\Omega_t=(1-t)\Omega_0+\Omega_1:=\{(1-t)x_0+tx_1:\ x_0\in \Omega_0,\ x_1\in \Omega_1 \},
\]
be open bounded Lipschitz domains in $\R^n$, $n\geq 2$.
Then
\[
\lambda_{p,\gamma}(\Omega_t)\leq (1-t)\lambda_{p,\gamma}(\Omega_0)+t\lambda_{p,\gamma}(\Omega_1).
\]
\end{theorem}
\begin{remark}\label{rem:Lip}
Note that, if $\Omega_0$ and $\Omega_1$ belong to the class $\mathcal{A}^n$ of Lipschitz domains in $\R^n$ with positive reach, then $\Omega_t$ belongs to the same class as well, as observed in \cite{Crasta-Fragal2020}. The Lipschitz regularity of the involved domains assures that the related solutions of problem \eqref{PDE1} are continuous up to the boundary. For this, of course, we could assume weaker conditions, in particular a suitable Wiener's type condition, see \cite{GariepyZiemer}.  On the other hand, in the proof of Theorem \ref{T1} the Lipschitz regularity of $\partial\Omega_t$ will be also used to prove that a viscosity sub-solution of problem \eqref{PDE1}, once extended as $0$ outside of $\Omega_t$, remains a viscosity sub-solution. For this property, as well, we could weaken the assumptions,
however, we prefer to give a clear and simple statement and to leave the investigation in this direction to a possible future research.
\end{remark}
\begin{remark}
We do not attempt to give an account of the vast literature concerning regularity for solutions of quasilinear elliptic equations.
We just point out (see Proposition \ref{prop:regularity}) that, for any bounded open domain $\Omega$, a weak solution $u$ of problem \eqref{PDE1} is in fact of class $C^{1,\alpha}_{\mathrm{loc}}(\Omega)\cap C^2(\Omega \setminus \bar{\C})$ for some $\alpha \in (0,1)$, where $\C=\{x\in \Omega:\ \nabla u(x)=0 \}$;
moreover, if $\partial\Omega$ is of class $C^{1,\alpha}$, then  $u\in C^{1,\beta}(\overline{\Omega})\cap C^2(\Omega \setminus \bar{\C})$ (some $\beta\in (0,1)$). Hence, if $\Omega_0$ and $\Omega_1$ are of class $C^{1,\alpha}$, taking into account also the previous Remark \ref{rem:Lip}, the proofs of Theorem \ref{T1} and of our other subsequent results would be simpler.
\end{remark}

As it is now well understood, when a Brunn-Minkowski type inequality holds, it can be used to derive an Urysohn's type inequality for the involved functional (see for instance \cite{BS}, and \cite{S} for a related general theory). In order to state the property that we obtain in this case, let us recall that, given an open bounded convex set $\Omega$ and a direction $y\in\sfe$, the {\em width of $\Omega$ in the direction $y$}, that we denote by $w(\Omega,y)$, is the distance between the supporting hyperplanes to $\Omega$ with outer unit normals $y$ and $-y$. The {\em mean width} $w(\Omega)$ of $\Omega$ is then defined as
$$
w(\Omega)=\frac1{\mathcal{H}^{n-1}(\sfe)}\int_{\sfe}w(\Omega,y)\d{\mathcal H}^{n-1}(y),
$$
where ${\mathcal H}^{n-1}$ is the $(n-1)$-dimensional Hausdorff measure.
If $\Omega$ is bounded, but not necessarily convex, we refer to its mean width $w(\Omega)$ as the mean width of its convex hull. Similarly to \cite[Corollary 1.5]{Andrea-Paolo2024}, we obtain the following result.

\begin{corollary}\label{cor-min} Let $p>1$ and $\Omega$ be an open bounded Lipschitz domain in $\R^n$, $n\ge 2$. Then
\begin{equation}\label{Urysohn}
\lambda_{p,\gamma}(\Omega)\geq \lambda_{p,\gamma}(\Omega^\sharp),
\end{equation}
where
$$
\mbox{$\Omega^\sharp$ is a ball centered at $0$ such that}\,\, w(\Omega^\sharp)=w(\Omega).
$$
\end{corollary}
We recall that in the plane mean width coincides with the Euclidean perimeter of the convex hull. Then the above corollary, for $n=2$, tells that the disk centered at the origin has the minimum principal Gauss $p$-frequency among sets with given Euclidean perimeter. Hence, on one hand, Corollary \ref{cor-min} (also for $n$ larger than $2$) may appear somewhat surprising, since  minimizers in the Faber-Krahn inequality in the Gauss space (like isoperimetric sets) are half-spaces (see \cite{Ehrhard-84} , and, for instance, \cite{BCF, CCLaMP} for more references). On the other hand, the mean width of a half-space is infinity, moreover the width is a metric quantity and does not depend on the measure, so \eqref{Urysohn} does not fall within the cohort of inherently Gaussian isoperimetric-type inequalities.

Our second main result is the extension to problem  \eqref{PDE1} of a well known result for the Laplacian: the first Dirichlet (positive) eigenfunction of the Laplace operator in a convex set is log-concave. This fact was proved by Brascamp and Lieb in \cite{Brascamp-Lieb1976}, and subsequently different proofs have been provided in \cite{Caffarelli-Friedman,Kennington,Korevaar1,Korevaar-Lewis}. The same result has been extended to the case of the $p$-Laplace operator in \cite{Sakaguchi1987} and to the Ornstein-Uhlenbeck operator \cite{Andrea-Paolo2024}. Here we prove the same for the Gaussian $p$-Laplacian.

\begin{theorem}\label{T2}
Let $p>1$, $n\geq 2$, $\Omega $ be an open, bounded and convex domain in $\R^n$, and $u$ be a solution of problem \eqref{PDE1}, with $u>0$ in $\Omega$.
Then the function
\[
W=\ln u
\]
is concave in $\Omega$.
\end{theorem}
\begin{remark}
We notice that under stronger regularity assumptions, it is possible to adapt the method of \cite{Sakaguchi1987} to obtain the log-concavity of positive first eigenfunctions, as it is done in the upcoming \cite{Lei2}. One distinguished feature of our approach, apart from allowing weaker regularity, is that we can treat both the logconcavity of eigenfunctions and the Brunn-Minkowski inequality of the principal frequency at once, with the very same argument.
\end{remark}

One of the crucial steps in the proofs of Theorems \ref{T1} and \ref{T2} is to consider a certain convolution of solutions of problem \eqref{PDE1} (of the solutions in the domains $\Omega_0$ and $\Omega_1$ for Theorem \ref{T1}, or of the solution in $\Omega$ with itself in the case of Theorem \ref{T2}, we mean). This convolution is shown to be a viscosity sub-solution, and then it is compared with the solution of the problem, by calculating its $p$-th Rayleigh Gaussian quotient.
For this, we need to prove that viscosity sub-solutions are also weak sub-solutions (see Theorem \ref{Viscosity-weak}). This fact has its own importance and it is, in our opinion, a third interesting result of this paper. Notice that the equivalence of distributional weak solutions and viscosity solutions for elliptic equations is a deep question, whose study was started by Lions \cite{Lions} and Ishii \cite{Ishii}. In the case of $p$-Laplace equation, the equivalence was established by Juutinen, Lindqvist and Manfredi \cite{Juutinen2001}. Then, Julin and Juutinen \cite{Versa-Julin2012} gave a new direct proof that applies to various other equations as well. Further generalizations can be found, for instance, in \cite{Radulescu2024, MO, Siltakoski2018}.
\smallskip

{  Finally, let us notice that, although the Gaussian measures is probably the most famous and studied one, it would be interesting to consider other different measures and to see what happens with weights that satisfy suitable concavity conditions. }

\medskip

The organization of the paper is as follows. In Section \ref{Se2}, we list some basic facts and notions. In Section \ref{Se3}, we prove that viscosity subsolutions are weak subsolutions of the equation in \eqref{PDE1}, see Theorem \ref{Viscosity-weak}. In Section \ref{Se4}, we prove Theorem \ref{T1} and Theorem \ref{T2}.

\medskip

{\bf Acknowledgments.} We thank Andrea Cianchi for pointing out Proposition \ref{prop:regularity}, and we thank both Andrea Cianchi and Matteo Focardi, from Universit\`a degli Studi di Firenze, for interesting discussions about the regularity of weak solutions of the problem at hand.

The first author was partially supported by INdAM through different GNAMPA projects, and by the Italian "Ministero dell'Universit\`a e Ricerca" and EU through different PRIN projects. The Third author has been partially financed by European Union -- Next Generation EU -- through the project "Geometric-Analytic Methods for PDEs and Applications (GAMPA)", within the PRIN 2022 program (D.D. 104 - 02/02/2022 Ministero dell'Universit\`a e della Ricerca).

\section{Preliminaries}\label{Se2}

\subsection{Basic notation}

Let $\mathbb{R}^n$ be the $n$-dimensional Euclidean space, $x$ is a point in $\R^n$, $|x|$ is the Euclidean norm of $x$ and $dx$ denotes integration with respect to Lebesgue measure in $\R^n$. For $x\in \R^n$ and $r>0$, we denote by $B_{r}(x)$ the open ball of radius $r$ centered at $x$. If $\Omega\subset \R^n$, we denote the closure, interior and the boundary of $\Omega$ by $\overline{\Omega}$, $\mathrm{int} \Omega$ and $\partial \Omega$, respectively, and we let $|\Omega|$ be its Lebesgue measure. Let $C_{0}^{\infty}(\Omega)$ be the set of functions from $C^{\infty}(\Omega)$ having compact support in $\Omega$. We denote the support set of a real-value function by $\mathrm{spt}(u)$. Let $I$ denote the $n\times n$ unit matrix and $\mathcal{S}^n$ denote the space of $n\times n$ symmetric matrices. We will often use $c$ for positive constants, which may vary between appearances, even within a chain of inequalities.

We denote the Gauss probability space by $( \mathbb{R}^{n}, \gamma)$, the measure $\gamma$ is given by
$$
\gamma(\Omega)= \frac{1}{(2\pi)^{n/2}} \int_{\Omega} e^{-|x|^2/2} dx,
$$
for any measurable set $\Omega\subseteq \mathbb{R}^{n}$. Throughout the paper, $d\gamma$ stands for integration with respect to $\gamma$, i.e., $d\gamma(x)=(2\pi)^{-n/2} e^{-|x|^2/2}dx$, and $d\gamma_{\partial \Omega}$ is the $(n-1)$-dimensional Hausdorff measure on $\partial\Omega$ with respect to $\gamma$.

Let $\Omega\subseteq \R^n$ be a measurable set. Given $1<p<+\infty$, the expressions $L^p(\Omega)$, $W^{1,p}(\Omega)$,  $W^{1,p}_{\text{loc}}(\Omega)$ and $W^{1,p}_0(\Omega)$ have the usual meaning. We denote by $L^p(\Omega,\gamma)$ the space of all measurable functions $u$ such that
\begin{equation}\nonumber
L^p(\Omega,\gamma)=\left\{ u:\Omega\rightarrow \R:\ \|u\|^{p}_{p,\gamma}:=\int_{\Omega} |u(x)|^p d\gamma <+\infty \right\}.
\end{equation}
Similarly,
$W^{1,p}(\Omega,\gamma)$, $W^{1,p}_{\mathrm{loc}}(\Omega,\gamma)$ and
$W^{1,p}_0(\Omega,\gamma)$ denotes the corresponding $\gamma$-weighted Sobolev spaces. When $\Omega$ is a bounded domain, then the density function $e^{-|x|^2/2}$ has positive upper and lower bounds in $\Omega$, hence it is obvious that $L^p(\Omega,\gamma)=L^p(\Omega)$, $W^{1,p}(\Omega,\gamma)=W^{1,p}(\Omega)$, and so on. However, we only have $W^{1,p}(\Omega)\subset W^{1,p}(\Omega,\gamma)$ and $W^{1,p}_0(\Omega)\subset W^{1,p}_0(\Omega,\gamma)$ when $\Omega$ is an unbounded domain. (see for instance \cite{Heinonen2006,Tero1994,Radulescu2019} and references therein).

If $u$ is twice differentiable, by $\nabla u$ and $D^2 u$ we denote the gradient of $u$ and its Hessian matrix, respectively, i.e., $\nabla u=(\frac{\partial u}{\partial x_1},\cdots,\frac{\partial u}{\partial x_n})$ and $D^2u=(\frac{\partial^2 u}{\partial x_i \partial x_j})$.

The Ornstein-Uhlenbeck operator is defined as
\begin{equation}\label{OU1}
L_\gamma(u):=\Delta u-(x,\nabla u)\,.
\end{equation}
For suitably regular functions and domains, clearly we have the following integration by parts identity:
\[
\int_{\Omega} \upsilon L_\gamma (u) d\gamma=-\int_\Omega \langle \nabla u, \nabla \upsilon \rangle d\gamma+\int_{\partial \Omega} \upsilon \langle \nabla u, n_x \rangle d\gamma_{\partial \Omega},
\]
where $n_x$ is the outward unit normal vector at point $x\in \partial \Omega$. In this paper, we consider an extension of the Ornstein-Uhlenbeck operator which can be considered the Gaussian version of the $p$-Laplace operator:
\begin{equation}
-\Delta_{p,\gamma} (u):= -\Delta_pu+(x,\nabla u)|\nabla u|^{p-2}\,,
\end{equation}
where $$
\Delta_pu=\mathrm{div} (|\nabla u|^{p-2}\nabla u)$$
is the usual $p$-Laplacian.
Similarly to the usual  Ornstein-Uhlenbeck operator (which corresponds to the case $p=2$), $\Delta_{p,\gamma}$ satisfies the following integration by parts identity
\[
\int_{\Omega} \upsilon \Delta_{p,\gamma}(u)  d\gamma=-\int_\Omega |\nabla u|^{p-2}\langle \nabla u, \nabla \upsilon \rangle d\gamma+\int_{\partial \Omega} \upsilon |\nabla u|^{p-2}\langle \nabla u, n_x \rangle d\gamma_{\partial \Omega},
\]
for suitably regular functions and domains.

As it is well known, the $p$-Laplace operator is degenerate elliptic for $p>2$ and singular for $1<p<2$. For this reason, it is natural to consider weak solutions of PDE's where the principal part is the $p$-Laplacian. However, it is often desirable to have a pointwise interpretation for identities involving the second derivatives of a function, even if these are only weak second derivatives and need not really exists everywhere.
So, for later convenience, we recall that for a generic function $u$, by a direct calculation, the following expression is valid at points where $u$ is twice differentiable and $\nabla u$ does not vanish:
\begin{equation}\label{p-Lapla-form}
\Delta_p u(x)=|\nabla u(x)|^{p-2} \Delta u(x)+(p-2)|\nabla u|^{p-4}\sum_{i,j=1}^n  \frac{\partial^2 u}{\partial x_i \partial x_j}\frac{\partial u}{\partial x_i} \frac{\partial u}{\partial x_j}  .
\end{equation}
Notice that when $p\geq 2$, the above expression is in fact formally valid also at critical points, and in particular, when $p>2$ it can be interpreted even if $u$ is just differentiable at $x$, in the sense that
$$
\Delta_pu(x)=0\qquad\text{when }\nabla u(x)=0\,\text{ and }p>2.
$$

\subsection{Weak solutions}
By a {\it weak solution} of problem \eqref{PDE1}, we mean a function $u\in W^{1,p}_{0}(\Omega,\gamma)$ such that, for every test function $\varphi\in C^{\infty}_0(\Omega)$, it holds
\begin{equation}\label{weak-solution}
\int_{\Omega} |\nabla u|^{p-2} \nabla u \cdot \nabla \varphi\ d\gamma -\lambda_{p,\gamma} \int_{\Omega}|u|^{p-2}u\cdot \varphi \ d\gamma= 0.
\end{equation}
Notice that the test function $\varphi$ in the inequality above, can be taken in fact in $W^{1,p}_{0}(\Omega,\gamma)$.

A function $u\in W^{1,p}_{\text{loc}}(\Omega,\gamma)$ is a {\it local weak super-solution} of the equation in \eqref{PDE1} if, for every non-negative function $\varphi\in C^{\infty}_0(\Omega)$, it holds
\begin{equation}\label{weak-supsolution}
\int_{\Omega} |\nabla u|^{p-2} \nabla u \cdot \nabla \varphi\ d\gamma -\lambda_{p,\gamma} \int_{\Omega}|u|^{p-2}u\cdot \varphi \ d\gamma\geq 0.
\end{equation}
Likewise, a {\it local weak sub-solution} of the equation in \eqref{PDE1} is defined as above with $\leq$ replacing $\geq$.
A function which is both a local weak sub-solution and a local weak super-solution is called {\it local weak solution}.


\begin{proposition}\label{prop:regularity}
If $\Omega$ is a bounded open domain and $u$ is a weak solution of problem \eqref{PDE1}, with $p>1$, then $u\in C^{1,\alpha}_{\mathrm{loc}}(\Omega)\cap C^2(\Omega\setminus\bar{\C})$ for some $\alpha \in (0,1)$, where $\C=\{x\in \Omega:\ \nabla u(x)=0 \}$.
Moreover, if $\partial\Omega$ is $C^{1,\alpha}$, then $u\in C^{1,\beta}(\overline\Omega)$, for some $\beta\in(0,1)$.
\end{proposition}
\begin{proof}
First we notice that $u$ is bounded, thanks for instance to Theorem 7.1 in Chapter 4 or to Theorem 3.1 in Chapter 5 of \cite{Ladizeskaya-Ural'ceva-1968}.
Then, we have $u\in C^{1,\alpha}_{\text{loc}}(\Omega)$ by \cite[Theorem 1]{Tolksdorf}. Now, in $\Omega\setminus\overline{\C}$, $u$ solves a linear equation of the type $-\text{div}(a(x) \nabla u)= f(x)$, with $a(x)$ and $f(x)$ H\"older continuous, hence $u\in C^{2,\alpha}$ by classic Schauder's theory (see for instance \cite{GT}). Moreover, if $\Omega$ is $C^{1,\alpha}$, then $u\in C^{1,\beta}(\overline\Omega)$ by \cite[Theorem 1]{Lieberman}.
\end{proof}
\begin{remark}\label{localtoglobalsolution}
Notice that, if $u\in W^{1,p}_{\mathrm{loc}}(\Omega,\gamma)\cap C^1(\overline\Omega)$ and $u=0$ on $\partial\Omega$, then $u\in W^{1,p}_{0}(\Omega,\gamma)$. Moreover, if it is a local weak solution of the equation in \eqref{PDE1}, then it is a weak solution of  problem \eqref{PDE1}. Thanks to Proposition $\ref{prop:regularity}$, this would make the proofs of our results simpler.
\end{remark}

\subsection{Viscosity solutions}
We give here basic concepts about viscosity solutions of elliptic equations. For more details, we refer to \cite{Crandall1995, Ishii-Lions1992, Katzourakis, Koike}, and in particular to \cite{Juutinen2001}.
A function $u:\Omega \rightarrow $ is {\it upper semi-continuous} in $\Omega$ if, for every $x\in\Omega$,
\[
u(x)=\mathop{\lim}_{r\rightarrow 0^+} \sup \{u(y):y\in \Omega,\ |y-x|<r\}.
\]
A function $u:\Omega \rightarrow $ is {\it lower semi-continuous} in $\Omega$ if, for every $x\in\Omega$,
\[
u(x)=\mathop{\lim}_{r\rightarrow 0^+} \inf \{u(y):y\in \Omega,\ |y-x|<r\}.
\]
We denote by $J^{2,-}_\Omega u (x)$ the second order sub-jet of $u$ at $x$, which is by definition the set of pairs $(\xi,A)\in \R^n \times \mathcal{S}^n$ such that, as $y \rightarrow x$, $y\in \overline{\Omega}$, it holds
\begin{equation}
u(y) \geq u(x)+(\xi, y-x)+\frac{1}{2}(A(y-x),y-x)+o(|y-x|^2).
\end{equation}
The closure of a sub-jet is defined by $(\xi,A)\in \bar{J}^{2,-}_\Omega u(x)$ if there exists a sequence $(\xi_j,A_j)\in J^{2,-}_\Omega u(x_j)$ such that $(x_j,u(x_j),\xi_j,A_j)\rightarrow (x,r,\xi,A)$ with some $r\in \R$ as $j\rightarrow \infty$. Obviously, $r=u(x)$ if $u$ is continuous. The super-jet $J^{2,+}_\Omega u(x)$ and $\bar{J}^{2,+}_\Omega u(x)$ are defined by a similar way with $\geq $ replacing $\leq $.

Let $u$ and $\varphi$ be two functions defined in $\Omega$, and let $x_0\in \Omega$. We say that {\it $\varphi$ touches $u$ from above} at $x_0$ if
\[
\varphi(x_0)=u(x_0)\ \ \text{and}\ \ \varphi(x)\geq u(x),
\]
in a neighbourhood of $x_0$ (i.e., $\varphi(x_0)=u(x_0)$ and $\varphi-u$ has a local minimum at $x_0$). Similarly, we say that {\it $\varphi$ touches $u$ from below} at $x_0$, if
\[
\varphi(x_0)=u(x_0)\ \ \text{and}\ \ \varphi(x)\leq u(x),
\]
in a neighbourhood of $x_0$ (i.e., $\varphi(x_0)=u(x_0)$ and $\varphi-u$ has a local maximum at $x_0$).

For $p>1$, set
\begin{equation}
F(x,u,\nabla u, D^2 u):=-\Delta_{p,\gamma}(u)-\lambda_{p,\gamma} |u|^{p-2}u=-\Delta_p u+(x,\nabla u)|\nabla u|^{p-2}-\lambda_{p,\gamma} |u|^{p-2}u.
\end{equation}
A function $u:\Omega \rightarrow (-\infty,\infty]$ is a {\it viscosity subsolution} to $F(x,u,\nabla u, D^2 u)=0$ in $\Omega$, if
\begin{enumerate}
  \item $u$ is upper semi-continuous;
  \item $u$ is not identically $+\infty$, and
  \item for every $C^2$ function $\varphi$ {\it{touching $u$ from above}} at any point $x\in \Omega$, with $\nabla \varphi (x)\neq 0$ if $p<2$, it holds
  \[
   F(x,\varphi(x),\nabla\varphi(x),D^2 \varphi(x))=-\Delta_{p,\gamma} \varphi(x)-\lambda_{p,\gamma} |\varphi|^{p-2}\varphi\leq 0;
   \]
or, equivalently, if for any $(\xi,A)\in {J}^{2,+}_\Omega u (x)$ (or $(\xi,A)\in \bar{J}^{2,+}_\Omega u (x)$), with $\xi\neq 0$ if $p<2$, it holds $F(x,u,\xi,A)\leq 0$.
\end{enumerate}

Similarly, a lower semicontinuous function $u$ is a {\it viscosity supersolution} of the equation $F=0$, if, for every $C^2$ function $\varphi$ {\it{touching $u$ from below}} at any point $x\in \Omega$, with $\nabla \varphi(x)\neq 0$ if $p<2$, it holds
\[
F(x,\varphi(x),\nabla\varphi(x),D^2 \varphi(x))\geq 0;
\]
or, equivalently, if for any $(\xi,A)\in {J}^{2,-}_\Omega u (x)$ (or $(\xi,A)\in \bar{J}^{2,-}_\Omega u (x)$), with $\xi\neq 0$ if $p<2$, it holds $F(x,u,\xi,A)\geq 0$.

A function $u\in C(\Omega)$ is a viscosity solution of \eqref{PDE1} if it is both a viscosity subsolution and a viscosity supersolution.

 We remark that in the case $p\geq2$, the requirement $\nabla \varphi(x) \neq 0$ (equivalently $\xi\neq 0$) is not in force.

Notice that in our assumptions a classical solution is always a viscosity solution and a viscosity solution is a classical solution if it is regular enough.

\section{Viscosity solutions are weak solutions }\label{Se3}

For the proofs of this section, we will follow the lines of the papers \cite{Radulescu2024, Versa-Julin2012, Juutinen2001, MO, Siltakoski2018}.

Given an open bounded domain $\Omega$, an exponent $q>1$ and a function $u:\Omega\to\R$, for $\varepsilon>0$ we define the sup-convolution $u_{\varepsilon}$ of $u$ as follows:
\begin{equation}\label{inf-convo}
u_{\varepsilon}(x)=\mathop{\sup}_{y\in\Omega} \left( u(y)-\frac{1}{q\varepsilon^{q-1}}|x-y|^q\right),\quad x\in \Omega.
\end{equation}
Similarly, we can define the inf-convolution. The properties of inf-convolutions are well known and can be found in several references, such as the ones cited above (see in particular \cite[Appendix A]{Versa-Julin2012}) and \cite{Katzourakis, Koike}, for instance. The sup-convolution has some specific properties, which we list it in the following proposition.
\begin{proposition}\label{inf-convo-prop}
Suppose that $u:\Omega \rightarrow \R$ is a bounded and upper semi-continuous function. Then the sup-convolution $u_\varepsilon$ satisfies the following properties:
\begin{itemize}
  \item The family $\{ u_\varepsilon\}$ is decreasing with respect to $\varepsilon$, $u_\varepsilon \geq u$ in $\Omega$ and $u_\varepsilon \rightarrow u$ locally uniformly as $\varepsilon \rightarrow  0$.

   \item $u_\varepsilon$ is locally Lipschitz and a.e. twice differentiable in the following sense: for almost every $x,y\in\Omega$,
      \[
      u_\varepsilon(y)=u_\varepsilon(x)+\nabla u_\varepsilon(x)\cdot (x-y)+\frac{1}{2}D^2 u_\varepsilon (x)(x-y)^2+o(|x-y|^2).
      \]

  \item There exists $r(\varepsilon )>0$ such that
  \[
  u_\varepsilon(x)=\mathop{\sup}_{y\in B_{r(\varepsilon)}(x)\cap \Omega} \left( u(y)-\frac{|x-y|^q}{q\varepsilon^{q-1}}\right), \]
  with $r(\varepsilon)\rightarrow 0 $ as $\varepsilon \rightarrow 0$.

  \item If $x\in \Omega_{r(\varepsilon)}:=\{ x\in \Omega: \mathrm{dist}(x,\partial \Omega)>r(\varepsilon)\}$, then there exists $x_\varepsilon\in B_{r(\varepsilon)}(x)$ such that
  \[
  u_{\varepsilon}(x)=u(x_\varepsilon)-\frac{1}{q\varepsilon^{q-1}}|x-x_\varepsilon|^q.
  \]
  Moreover, if the gradient $\nabla u_\varepsilon(x)$ exists, it holds
  \[
  \left( \frac{|x-x_\varepsilon|}{\varepsilon}\right)^{q-1} \leq |\nabla u_\varepsilon(x)|.
  \]
  In particular, if $\nabla u_\varepsilon (x)=0$, then $u_\varepsilon (x)=u(x)$.

  \item If $(\eta,X)\in J^{2,+} u_\varepsilon(x)$ with $x\in \Omega_{r(\varepsilon)}$, then
  \[
  \eta=\frac{|x-x_\varepsilon|^{q-2}(x-x_\varepsilon)}{\varepsilon^{q-1}},\ \mathop{and}\ X \geq - \frac{q-1}{\varepsilon}|\eta|^{\frac{q-2}{q-1}} I;
  \]
 here $I$ is the identity matrix of order $n$.
  \item $u_\varepsilon $ is semi-convex in $\Omega_{r(\varepsilon)}$, that is, there is a positive constant $c$ such that the function
  \[
  x\mapsto u_\varepsilon (x)+c|x|^2
  \]
  is convex, where the constant $c$ depends on $\varepsilon,\,q$ and the oscillation $\sup_{\Omega}-\inf_{\Omega} u$. Moreover, for a.e. $x\in \Omega_{r(\varepsilon)}$, $D^2 u_\varepsilon(x) \geq -cI$.
\end{itemize}
\end{proposition}

Throughout the paper, for $p>1$, we fix $q\geq \max\{2,p/{(p-1)}\}$, so that we can take $q=p/{(p-1)}$ when $1<p<2$, and $q=2$ when $p\geq 2$.

\subsection{An approximating equation for $u_\epsilon$.}
\begin{lemma}\label{Viscosity1}
Given $p>1$. Suppose that $u:\Omega \rightarrow \R$ is bounded, non-negative and upper semi-continuous function. Then, if $u$ is a viscosity sub-solution in $\Omega$ of the equation \eqref{PDE1}, then $u_\varepsilon$ is a viscosity sub-solution in $ \Omega_{r(\varepsilon)}$ of
\begin{equation}\label{viscosity-super-solution}
 -\Delta_p u +(x,\nabla u)|\nabla u|^{p-2}= f_\varepsilon(x,u,\nabla u),
\end{equation}
where
\begin{equation}\label{3-1}
f_\varepsilon(x,u_\varepsilon,\nabla u_\varepsilon):=\varepsilon |\nabla u_\varepsilon|^{p}+\lambda_{p,\gamma}\cdot{\sup}_{y\in B_{r(\varepsilon)}(x)} u_\varepsilon(y)^{p-1}.
\end{equation}
\end{lemma}

\begin{proof}[\bf Proof.]
First set $q=\max\{2,p/(p-1)\}$, i.e. $q=2$ if $p\geq 2$ and $q=p/(p-1)$ if $p\in(1,2)$.
Let $\varphi$ be a $C^2$ function touching $u_\varepsilon$ from above at a point $x^0\in \Omega_{r(\varepsilon)}$, that is,
\begin{equation}\label{3-4}
\varphi(x^0)=u_\varepsilon(x^0)\ \ \text{and}\ \ \varphi(x)\geq u_\varepsilon(x),
\end{equation}
in a neighbourhood of $x^0$, and $\nabla \varphi(x^0)\neq 0$. Set
\begin{equation}\label{3-5}
\eta=\nabla\varphi(x^0),\quad X=D^2 \varphi(x^0).
\end{equation}
Then $(\eta,X)\in J^{2,+}u_\varepsilon (x^0)$ and, by Proposition \ref{inf-convo-prop}, there exists $x^0_\varepsilon\in B_{r(\varepsilon)}(x^0)$ such that
\begin{equation}\label{3-2}
u_{\varepsilon}(x^0)=u(x^0_\varepsilon)-\frac{1}{q\varepsilon^{q-1}}|x^0-x^0_\varepsilon|^q ,
\end{equation}
and
\begin{equation}\label{3-3}
\eta=\frac{|x^0-x^0_\varepsilon|^{q-2}(x^0-x^0_\varepsilon)}{\varepsilon^{q-1}},\ \ X\geq - \frac{q-1}{\varepsilon^{q-1}}|x^0-x^0_\varepsilon|^{q-2} I.
\end{equation}
Consider the function
\[
\Phi(y,z)=\frac{|y-z|^q}{q\varepsilon^q},\quad y,z\in \Omega_{r(\varepsilon)} \times \Omega_{r(\varepsilon)}.
\]
By a direct calculation, we have
$$
-D_y \Phi(x^0_\varepsilon,x^0)=D_z \Phi(x_\varepsilon,x^0)=\eta,
$$
and
\[
B:=D_{yy} \Phi(x^0_\varepsilon,x^0)=\varepsilon^{1-q}|x^0_\varepsilon-x^0|^{q-4}[|x^0_\varepsilon-x^0|^2I +(q-2)(x^0_\varepsilon-x^0)\otimes (x^0_\varepsilon-x^0)],
\]
along with $D_{zz} \Phi(x^0_\varepsilon,x^0)=B$ and $D_{yz} \Phi(x^0_\varepsilon,x^0)=D_{zy} \Phi(x^0_\varepsilon,x^0)=-B$.

By \cite[Theorem 3.2]{Ishii-Lions1992}, there exist symmetric matrices $Y,Z$ such that
\begin{equation}\label{YZ}
(\eta, -Y)\in \bar{J}^{2,+}u(x^0_\varepsilon),\quad (\eta,-Z)\in \bar{J}^{2,-}\varphi (x^0),
\end{equation}
and
\[
\begin{pmatrix} -Y & 0 \\ 0 & Z\end{pmatrix} \leq D^2 \Phi(x^0_\varepsilon,x^0) +\varepsilon^{1-q}(D^2 \Phi(x^0_\varepsilon,x^0))^2,
\]
where
\begin{equation}\nonumber
D^2\Phi(x^0_\varepsilon,x_0) =
 \begin{pmatrix} D_{yy} \Phi(x^0_\varepsilon,x^0) & D_{yz} \Phi(x^0_\varepsilon,x^0) \\ D_{zy} \Phi(x^0_\varepsilon,x^0) & D_{zz} \Phi(x^0_\varepsilon,x^0)\end{pmatrix}.
\end{equation}
Thus,
\begin{equation}\nonumber
\begin{pmatrix} -Y & 0 \\ 0 & Z\end{pmatrix} \leq \begin{pmatrix} B & -B \\ -B & B\end{pmatrix}+ 2\varepsilon^{1-q} \begin{pmatrix} B^2 & -B^2 \\ -B^2 & B^2\end{pmatrix} .
\end{equation}
It follows $Y \geq Z$. Moreover, $-Z \leq X$ by \eqref{3-5} and \eqref{YZ}.

Thanks to \eqref{p-Lapla-form}, and using the above facts, we have
\begin{align*}
-\Delta_p \varphi (x^0) &=-|\nabla \varphi(x^0)|^{p-2} \Delta \varphi(x^0)-(p-2)|\nabla \varphi|^{p-4}\sum_{i,j=1}^n  \frac{\partial^2 \varphi}{\partial x_i \partial x_j}\frac{\partial \varphi}{\partial x_i} \frac{\partial \varphi}{\partial x_i}\\
&= -|\eta|^{p-2} tr(X)-(p-2)|\eta|^{p-4}(X\eta,\eta)\\
& \leq |\eta|^{p-2} tr(Z)+(p-2)|\eta|^{p-4}(Z\eta,\eta)\\
& \leq |\eta|^{p-2} tr(Y)+(p-2)|\eta|^{p-4}(Y\eta,\eta).
\end{align*}
Since $(\eta, -Y)\in \bar{J}^{2,+}u(x^0_\varepsilon)$, and $u$ is a viscosity sub-solution of the equation in \eqref{PDE1}, we have
\begin{equation}\nonumber
|\eta|^{p-2} tr(Y)+(p-2)|\eta|^{p-4}(Y\eta,\eta)+(x^0_\varepsilon,\eta)|\eta|^{p-2}-\lambda_{p,\gamma}u(x^0_\varepsilon)^{p-1}\leq 0.
\end{equation}
Therefore,
\begin{equation}\nonumber
-\Delta_p \varphi (x^0) +(x^0_\varepsilon,\eta)|\eta|^{p-2}-\lambda_{p,\gamma}u(x^0_\varepsilon)^{p-1} \leq 0.
\end{equation}
We can rewrite the latter inequality as follows:
\begin{equation}\nonumber
-\Delta_p \varphi (x^0)+(x^0,\eta)|\eta|^{p-2}  \leq -(x^0_\varepsilon-x^0,\eta)|\eta|^{p-2}+\lambda_{p,\gamma}u(x^0_\varepsilon)^{p-1}.
\end{equation}

Observe that
\[
 -(x^0_\varepsilon-x^0,\eta)|\eta|^{p-2}=\varepsilon |\eta|^p.
\]
On the other hand, by the definition of $u_\varepsilon$, we know that $u_\varepsilon \geq u$. By the fact that $x^0_\varepsilon\in B_{r(\varepsilon)}(x^0)$ and \eqref{3-4}, we have
\[
u(x^0_\varepsilon) \leq u_\varepsilon (x^0_\varepsilon)\leq {\sup}_{y\in B_{r(\varepsilon)}(x^0)} u_\varepsilon(y)\leq {\sup}_{y\in B_{r(\varepsilon)}(x^0)} \varphi(y).
\]
Therefore,
we get that
\[
-\Delta_p \varphi (x^0) +(x^0,\nabla \varphi)|\nabla \varphi|^{p-2} \varphi (x^0)\leq f_\varepsilon(x^0,\varphi,\nabla \varphi).
\]
This concludes the proof.
\end{proof}

\subsection{A Caccioppoli's estimate.}

In the next Lemma, we provide a Caccioppoli's estimate for the $L^p_{\mathrm{loc}}$-norm of the gradients of $u_\varepsilon$.

\begin{lemma}\label{Caccioppoli}
Under the same assumptions of Lemma \ref{Viscosity1}, {   If, for every nonnegative test function $\varphi\in C^{\infty}_0(\Omega)$, $u_\varepsilon$ satisfies the following inequality
\begin{equation}\label{newassumption}
\int_{\Omega}|\nabla u_\varepsilon|^{p-2}\nabla u_\varepsilon\cdot \nabla \varphi d\gamma \leq \int_\Omega  f_\varepsilon(x,u_\varepsilon,\nabla u_\varepsilon)\varphi d\gamma \,,
\end{equation}}
then, there exists a positive constant $c$, depending on $p,\,\lambda_{p,\,\gamma},\|u\|_{L^{\infty}(\Omega)},|\Omega|$,  such that for every test function $\psi\in C^{\infty}_0(\Omega_{r(\varepsilon)})$, $ 0\leq \psi \leq 1$, we have
\begin{equation}\label{Caccioppoli ineq}
\int_{\Omega} |\nabla u_\varepsilon|^p \psi^p d\gamma \leq c(1+\|\nabla \psi\|_{L^{\infty}(\Omega)}),
\end{equation}
for sufficiently small $\varepsilon>0$.
\end{lemma}

\begin{proof}[\bf Proof.]
Let $\psi \in C^{\infty}_0(\Omega_{r(\varepsilon)})$, $0\leq\psi\leq1$ and consider the test function
\[
\varphi(x):=(u_\varepsilon(x)-{\inf}_{K} u_\varepsilon  )\psi^p(x),\quad x\in \Omega.
\]
By assumption, we have
\[
\int_{\Omega}|\nabla u_\varepsilon|^{p-2}\nabla u_\varepsilon\cdot \nabla \varphi d\gamma \leq \int_\Omega  f_\varepsilon(x,u_\varepsilon,\nabla u_\varepsilon)\varphi d\gamma .
\]
By a direct calculation, the integral on the left hand side equals to
\begin{equation}\label{333}
\int_{\Omega} |\nabla u_\varepsilon|^{p-2}\nabla u_\varepsilon \cdot[ \psi^p \nabla u_\varepsilon+p\psi^{p-1}\nabla \psi ( u_\varepsilon(x)-{\inf}_{K} u_\varepsilon)]d\gamma.
\end{equation}
The Young's inequality
\[
ab \leq \delta a^q+\delta^{-1/(q-1)} b^{q'}
\]
where $\delta>0$, $q$ and $q'$ are conjugate exponents such that $1/q+1/q'=1$, implies that the absolute value of second term in \eqref{333} is bounded by
\[
\delta \int_\Omega |\nabla u_\varepsilon|^p \psi^p +\delta^{1-p} \int_\Omega p^p |\nabla \psi|^p (\mathrm{osc}_K u_\varepsilon)^pd\gamma.
\]
Therefore, we have
\begin{align}\label{3-6}
(1-\delta)\int_\Omega |\nabla u_\varepsilon|^p \psi^pd\gamma  \leq  \delta^{1-p}\int_\Omega p^p |\nabla \psi|^p (\mathrm{osc}_K u_\varepsilon)^pd\gamma+\int_\Omega  f_\varepsilon(x,u_\varepsilon,\nabla u_\varepsilon)\varphi d\gamma.
\end{align}
Observe that $\max\{|u_\varepsilon|, \mathrm{osc}_K u_\varepsilon\}\leq \|u\|_{L^{\infty}(\Omega)}$. By Proposition \ref{inf-convo-prop}, we have
\begin{align*}
\int_\Omega  f_\varepsilon(x,u_\varepsilon,\nabla u_\varepsilon)\varphi d\gamma&=\int_\Omega  (\varepsilon |\nabla u_\varepsilon|^p+\lambda_{p,\gamma}\cdot{\sup}_{y\in B_{r(\varepsilon)}(x)} u_\varepsilon(y)^{p-1} )\varphi\,d\gamma\\
& \leq \varepsilon\|u\|_{L^{\infty}(\Omega)}\int_\Omega |\nabla u_\varepsilon|^p  \psi^p  d\gamma+\lambda_{p,\gamma}\|u\|^p_{L^{\infty}(\Omega)}\cdot |\gamma(\Omega)|.
\end{align*}
Selecting $\delta <1/3$ and $\varepsilon\leq \frac{1}{3}\|u\|^{-1}_{L^{\infty}(\Omega)}$, by \eqref{3-6}, we have
\[
\int_\Omega |\nabla u_\varepsilon|^p \psi^pd\gamma  \leq c(1+\|\nabla \psi\|_{L^{\infty}(\Omega)}),
\]
where $c$ depends on $p,\lambda_{p,\gamma},\|u\|_{L^{\infty}(\Omega)},|\Omega|$, as desired.
\end{proof}

\begin{lemma}\label{weak-converges}
Under the assumptions of Lemma \ref{Viscosity1}{    and Lemma \ref{Caccioppoli}}, we have $u\in W^{1,p}_{\mathrm{loc}}(\Omega,\gamma)$, and $\nabla u_\varepsilon \rightharpoonup \nabla u$ weakly in $L^{p}(K)$ for every compact set  $K \subset\Omega$.
\end{lemma}

\begin{proof}[\bf Proof.]
The just proved Caccioppoli's estimate \eqref{Caccioppoli ineq} applied to $u_\varepsilon$ in \eqref{viscosity-super-solution}, allows us to conclude that
\[
|\nabla u_\varepsilon|^{p-2}\nabla u_\varepsilon
\]
converges weakly in $L^{p/(p-1)}_{\mathrm{loc}}(\Omega,\gamma)$. Indeed, for any compact set $K\subset \Omega$, choose
an open set $U\subset \Omega$ containing $K$ and
a non-negative test function $0\leq \psi \leq 1$ such that
\[
K\subset K':=\mathrm{spt} \psi \subset U,
\]
and $\psi=1$ in $K$. Then
\[
\int_K ||\nabla u_\varepsilon|^{p-2}\nabla u_\varepsilon |^{p/(p-1)} d\gamma= \int_K |\nabla u_\varepsilon|^{p}d\gamma \leq \int_\Omega |\nabla u_\varepsilon|^p \psi^p d\gamma.
\]
So, by Lemma \ref{Caccioppoli}, we can find an uniform bound for the integrals
\[
\int_K ||\nabla u_\varepsilon|^{p-2}\nabla u_\varepsilon |^{p/(p-1)} d\gamma,\quad \int_\Omega |\nabla u_\varepsilon|^p d\gamma.
\]
Hence, $|\nabla u_\varepsilon|^{p-2}\nabla u_\varepsilon$ converges weakly in $L_{\mathrm{loc}}^{p/(p-1)}(\Omega,\gamma)$, and $\nabla u_\varepsilon $ converges weakly in $L_{\mathrm{loc}}^p(\Omega,\gamma )$. By Proposition \ref{inf-convo-prop}, we know that $u_\varepsilon$ converges pointwise to $u$. Therefore, we infer that $u\in W^{1,p}_{\mathrm{loc}}(\Omega,\gamma)$, and $u_\varepsilon$ converges weakly in $W^{1,p}_{\mathrm{loc}}(\Omega,\gamma)$.
\end{proof}

\begin{lemma}\label{error-term}
Under the assumptions of Lemma \ref{Viscosity1} {    and Lemma \ref{Caccioppoli}}, for any function $\psi \in C^{\infty}_0(\Omega)$,
\begin{equation}\nonumber
\mathop{\lim}_{\varepsilon \rightarrow 0} \int_{\Omega} f_\varepsilon(x,u_\varepsilon,\nabla u_\varepsilon) \psi d\gamma=\lambda_{p,\gamma}\int_\Omega u^{p-1}  \psi d\gamma.
\end{equation}
\end{lemma}

\begin{proof}[\bf Proof.] Let $\psi \in C^{\infty}_0(\Omega)$ and set $K:=\mathrm{spt}(\psi)$. Since we can write $\psi=\psi^{+}-\psi^{-}$, it is enough to prove the statement for $\psi\geq 0$. Consider $\varepsilon>0$ small enough so that
\[
K \subset K' \subset \Omega,
\]
where $K':= \overline{\cup_{x\in K}B_{r(\varepsilon)}(x)}$ and $r(\varepsilon)$ is given by Proposition \ref{inf-convo-prop}.

By Lemma \ref{weak-converges}, there exists a uniform upper bound of $\nabla u_\varepsilon$ in $L^{p}(K,\gamma)$, that is,
\begin{equation}\label{3-8}
\int_{\Omega}  |\nabla u_\varepsilon|^{p}\psi d\gamma \leq c,
\end{equation}
for all $\varepsilon \in (0,\varepsilon_0)$, where $c>0$ is a constant which is independent of $\varepsilon$. By the definition of $u_\varepsilon$, we have
\begin{equation}\label{3-9}
|u_\varepsilon|\leq \|u\|_{L^{\infty}(\Omega)}<+\infty.
\end{equation}
{  By
Proposition \ref{inf-convo-prop}, we have
\begin{equation}\label{3-10}
\mathop{\lim}_{\varepsilon \rightarrow 0}  \left({\sup}_{y\in B_{r(\varepsilon)}(x)} u_\varepsilon(y)^{p-1}\right) = u(x)^{p-1}\quad \mathrm{in}\ \Omega.
\end{equation}}
Hence, by the definition of $f_\varepsilon$, and \eqref{3-8}, \eqref{3-9}, \eqref{3-10} and Lebesgue dominated convergence theorem, we obtain
\[
\mathop{\lim}_{\varepsilon \rightarrow 0} \int_{\Omega} f_\varepsilon(x,u_\varepsilon,\nabla u_\varepsilon) \psi d\gamma=\lambda_{p,\gamma} \int_{\Omega}u^{p-1}\psi d\gamma.
\]
\end{proof}

\subsection{Viscosity solutions are weak solutions}

\begin{theorem}\label{Viscosity-weak}
Let $p>1$ and assume $u:\Omega \rightarrow \R$ is bounded, non-negative and upper semi-continuous. If $u$ is a viscosity sub-solution of the equation in \eqref{PDE1}, then it is also a local weak sub-solution of the same equation.
\end{theorem}

\begin{proof}[\bf Proof ]
The case $p=2$ is treated in \cite{Andrea-Paolo2024} and regarding in particular this theorem one can refer to \cite{Ishii}. We will then consider $p\neq 2$ only.

{\bf Case 1. (degenerate case $p> 2$).} Let  $u_{\varepsilon}$ be the standard sup-convolution, defined in \eqref{inf-convo} with $q=2$.
From Proposition \ref{inf-convo-prop},
we know that $\{u_\varepsilon\}$ is a decreasing family of semi-convex viscosity sub-solutions to \eqref{viscosity-super-solution} in $\Omega_{r(\varepsilon)}$, which converge pointwise to $u$, as $\varepsilon\rightarrow 0$. In particular, the function
\begin{equation}\label{varphi}
x\mapsto \varphi_\varepsilon(x):=u_\varepsilon (x)+\frac{1}{2\varepsilon}|x|^2
\end{equation}
is convex in $\Omega_{r(\varepsilon)}$.

By Aleksandrov's theorem, $u_\varepsilon$ is twice differentiable a.e. in $\Omega_{r(\varepsilon)}$, hence formula \eqref{p-Lapla-form} holds at any point of twice differentiability and, owing to Lemma \ref{Viscosity1}, we have
\begin{equation}\label{4-9-1}
-\Delta_p u_\varepsilon +(x,\nabla u_\varepsilon(x))|\nabla u_\varepsilon(x)|^{p-2}\leq f_\varepsilon(x,u_\varepsilon,\nabla u_\varepsilon),
\end{equation}
almost everywhere in $\Omega_{r(\varepsilon)}$.
Now, let us fix any non-negative function $\psi \in C^{\infty}_0(\Omega)$ and let $\varepsilon$ be small enough so that $\text{spt}(\psi)\subset\Omega_{r(\varepsilon)}$. Next we show {  \eqref{newassumption}, that is}
\begin{equation}\label{4-10}
\int_{\Omega} |\nabla u_\varepsilon|^{p-2} \nabla {u_\varepsilon} \cdot \nabla\psi d\gamma-\int_{\Omega}f_\varepsilon(x,u_\varepsilon,\nabla u_\varepsilon) \psi d\gamma \leq 0\,.
\end{equation}
Indeed, let $\varphi_j$ be a sequence of smooth convex functions, obtained via standard mollification, converging to the function $\varphi_\varepsilon$ defined in \eqref{varphi}, and set $u_{\varepsilon,j}=\varphi_j-\frac{1}{2\varepsilon}|x|^2$. We observe that integration by parts gives
\[
\int_{\Omega} \left[|\nabla u_{\varepsilon,j}|^{p-2} \nabla u_{\varepsilon,j} \cdot \nabla\psi -(x,\nabla u_{\varepsilon,j})|\nabla u_{\varepsilon,j}|^{p-2} \cdot \psi\right] d\gamma=\int_\Omega (-\Delta_p u_{\varepsilon,j})\psi d\gamma.
\]
Since $u_\varepsilon$ is locally Lipschitz continuous, we have
\[
\int_\Omega |\nabla u_\varepsilon|^{p-2} \nabla u_\varepsilon \cdot \nabla \psi d\gamma = \mathop{\lim}_{j\rightarrow \infty} \int_\Omega |\nabla u_{\varepsilon,j}|^{p-2} \nabla u_{\varepsilon,j} \cdot \nabla\psi d\gamma,
\]
and
\[
\int_{\Omega} (x,\nabla u_\varepsilon)|\nabla u_\varepsilon|^{p-2}\psi d\gamma = \mathop{\lim}_{j\rightarrow \infty} \int_{\Omega} (x,\nabla u_{\varepsilon,j})|\nabla u_{\varepsilon,j}|^{p-2}\psi d\gamma.
\]
On the other hand, since $D^2 u_{\varepsilon,j}\geq- \frac{1}{\varepsilon} I$ and $\nabla u_{\varepsilon,j}$ is locally bounded, we have
\[
\Delta_p u_{\varepsilon,j} \geq -\frac{c_1}{\varepsilon},
\]
in the support set of $\psi$, where $c_1>0$ is a constant depending only on $p$ and $n$. Thus, by Fatou's Lemma, we have
\[
\mathop{\lim \inf}_{j\rightarrow \infty} \int_\Omega\Delta_p u_{\varepsilon,j} \psi d\gamma \geq  \int_\Omega \mathop{\lim \inf}_{j\rightarrow \infty}\Delta_p u_{\varepsilon,j} \psi d\gamma.
\]
Moreover, it is shown in \cite[p.242]{Evans1992}, that $D^2\varphi_j(x)\rightarrow D^2\varphi(x)$ for a.e. $x$, and then
\[
\mathop{\lim \inf }_{j\rightarrow \infty} \Delta_p u_{\varepsilon,j} (x)=\Delta_p u_\varepsilon (x),
\]
for a.e $x\in \Omega$.

Finally, using integration by parts, convergence properties above and \eqref{4-9-1}, we have
\begin{align*}
&\int_{\Omega} \left[|\nabla u_{\varepsilon}|^{p-2} \nabla u_{\varepsilon} \cdot \nabla \psi -(x,\nabla u_{\varepsilon})|\nabla u_{\varepsilon}|^{p-2}\right] \cdot \psi d\gamma \\
&\quad=\mathop{\lim}_{j\rightarrow \infty}\int_{\Omega} \left[|\nabla u_{\varepsilon,j}|^{p-2} \nabla u_{\varepsilon,j} \cdot \nabla \psi -(x,\nabla u_{\varepsilon,j})|\nabla u_{\varepsilon,j}|^{p-2}\right] \cdot \psi d\gamma\\
&\quad =-\mathop{\lim }_{j\rightarrow \infty} \int_{\Omega} \Delta_p u_{\varepsilon,j} \psi d\gamma\\
&\quad \leq -\int_{\Omega}\mathop{\lim \inf}_{j\rightarrow \infty} \Delta_p u_{\varepsilon,j} \psi d\gamma\\
&\quad =\int_\Omega (-\Delta_p u_\varepsilon) \psi d\gamma\\
&\quad \leq -\int_\Omega (x,\nabla u_\varepsilon(x))|\nabla u_\varepsilon(x)|^{p-2}\psi d\gamma-\int_{\Omega}f_\varepsilon \psi d\gamma.
\end{align*}
Therefore, \eqref{4-10} holds. By Lemma \ref{weak-converges} and Lemma \ref{error-term}, letting $\varepsilon \rightarrow 0$ in \eqref{4-10}, we have
\[
\int_\Omega |\nabla u|^{p-2} \nabla {u} \cdot \nabla \psi d\gamma - \lambda_{p,\gamma} \int_{\Omega} |u|^{p-2} u \psi d\gamma\leq 0\,.
\]
For the arbitrariness of $\psi \in C^{\infty}_0(\Omega)$, this completes the proof.

{\bf Case 2. (Singular case $1<p<2$).} Let $u_\varepsilon$ be the sup-convolution of $u$, defined in \eqref{inf-convo} with $q>p/{(p-1)}$.
As in the degenerate case, we know from Proposition \ref{inf-convo-prop}
that $\{u_\varepsilon\}$ is a decreasing family of semi-convex viscosity sub-solutions to \eqref{viscosity-super-solution} in $\Omega_{r(\varepsilon)}$, which converge pointwise to $u$, as $\varepsilon\rightarrow 0$. In particular, there is $c$, depending on $\varepsilon,\,q$ and the oscillation $\sup_{\Omega}-\inf_{\Omega} u$, such that the function
\begin{equation}\label{varphiq}
x\mapsto \varphi_\varepsilon(x):=u_\varepsilon (x)+c|x|^2
\end{equation}
is convex in $\Omega_{r(\varepsilon)}$.
Hence, by Aleksandrov's theorem, $u_\varepsilon$ is twice differentiable a.e., and by Lemma \ref{Viscosity1}, we have
\begin{align}\label{3-11}
-\Delta_{p,\gamma} u_\varepsilon&=-\mathrm{div}(|\nabla u_\varepsilon|^{p-2}\nabla u_\varepsilon)+(x,\nabla u_\varepsilon)|\nabla u_\varepsilon|^{p-2}\nonumber\\
&=-|\nabla u_\varepsilon|^{p-2} \Delta u_\varepsilon-(p-2)|\nabla u_\varepsilon|^{p-4}\sum_{i,j=1}^n  \frac{\partial^2 u}{\partial x_i \partial x_j}\frac{\partial u_\varepsilon}{\partial x_i} \frac{\partial u_\varepsilon}{\partial x_j}+(x,\nabla u_\varepsilon)|\nabla u_\varepsilon|^{p-2}\nonumber\\
&\leq f_\varepsilon(x,u_\varepsilon,\nabla u_\varepsilon),
\end{align}
a.e. in $\Omega_{r(\varepsilon)}\setminus \{\nabla u_\varepsilon=0 \}$.

First, we show that
\begin{equation}\label{3-12}
\int_\Omega (|\nabla u_\varepsilon|^2+\delta)^{\frac{p-2}{2}}( \nabla u_\varepsilon \cdot \nabla \psi-(x,\nabla u_\varepsilon)\psi) d\gamma \leq \int_{\Omega}-\mathrm{div} \left((|\nabla u_\varepsilon|^2+\delta)^{\frac{p-2}{2}} \nabla u_\varepsilon\right)  \psi d\gamma
\end{equation}
for any non-negative $\psi\in C^{\infty}_0(\Omega_{r(\varepsilon)})$, and for every $\delta>0$ (adding the constant $\delta$ is necessary due to the singularity).
Indeed, let us fix a non-negative function $\psi \in C^{\infty}_0(\Omega_{r(\varepsilon)})$, and let $\varphi_j$ be a sequence of smooth convex functions converging to $\varphi_\varepsilon$ defined in \eqref{varphiq}, obtained via standard mollification. We denote $u_{\varepsilon,j}=\varphi_j-c|x|^2$ and observe that  an integration by parts gives
\[
\int_{\Omega}  (|\nabla u_{\varepsilon,j}|^2+\delta)^{\frac{p-2}{2}} (\nabla u_{\varepsilon,j} \cdot \nabla\psi  -(x,\nabla u_{\varepsilon,j}) \psi) d\gamma=-\int_\Omega \mathrm{div} \left((|\nabla u_{\varepsilon,j}|^2+\delta)^{\frac{p-2}{2}} \nabla u_{\varepsilon,j}\right) \psi d\gamma.
\]
Since $u_\varepsilon$ is locally Lipschitz continuous, we have
\[
\int_\Omega (|\nabla u_{\varepsilon}|^2+\delta)^{\frac{p-2}{2}} \nabla u_{\varepsilon} \cdot \nabla \psi d\gamma = \mathop{\lim}_{j\rightarrow \infty} \int_\Omega (|\nabla u_{\varepsilon,j}|^2+\delta)^{\frac{p-2}{2}} \nabla u_{\varepsilon,j} \cdot \nabla\psi d\gamma,
\]
and
\[
\int_{\Omega} (x,\nabla u_\varepsilon)(|\nabla u_{\varepsilon}|^2+\delta)^{\frac{p-2}{2}}\psi d\gamma = \mathop{\lim}_{j\rightarrow \infty} \int_{\Omega} (x,\nabla u_{\varepsilon,j})(|\nabla u_{\varepsilon,j}|^2+\delta)^{\frac{p-2}{2}}\psi d\gamma.
\]
On the other hand, since $D^2 u_{\varepsilon,j}\geq -2c I$ and $\nabla u_{\varepsilon,j}$ is locally bounded, we have
\[
\mathrm{div} \left((|\nabla u_{\varepsilon,j}|^2+\delta)^{\frac{p-2}{2}} \nabla u_{\varepsilon,j}\right)  \geq -c_2,
\]
in the support set of $\psi$, where $c_2>0$ is a constant depending only on $p,n,c$. Thus, by Fatou's Lemma, we have
\[
\mathop{\lim \inf}_{j\rightarrow \infty} \int_\Omega \mathrm{div} \left((|\nabla u_{\varepsilon,j}|^2+\delta)^{\frac{p-2}{2}} \nabla u_{\varepsilon,j}\right)  \psi d\gamma \geq  \int_\Omega \mathop{\lim \inf}_{j\rightarrow \infty}\left(\mathrm{div} ((|\nabla u_{\varepsilon,j}|^2+\delta)^{\frac{p-2}{2}} \nabla u_{\varepsilon,j}) \right) \psi d\gamma.
\]
Moreover, it is shown in \cite[p.242]{Evans1992}, that $D^2\varphi_j(x)\rightarrow D^2\varphi(x)$ for a.e. $x\in \Omega$, and then
\[
\mathop{\lim \inf }_{j\rightarrow \infty} \mathrm{div} ((|\nabla u_{\varepsilon,j}|^2+\delta)^{\frac{p-2}{2}} \nabla u_{\varepsilon,j})=\mathrm{div} \left((|\nabla u_\varepsilon|^2+\delta)^{\frac{p-2}{2}} \nabla u_\varepsilon\right) ,
\]
for a.e $x\in \Omega$. Finally, using again the integration by parts, convergence properties above and \eqref{4-9}, we have
\begin{align*}
&\int_\Omega (|\nabla u_\varepsilon|^2+\delta)^{\frac{p-2}{2}}( \nabla u_\varepsilon \cdot \nabla \psi -(x,\nabla u_{\varepsilon}) \psi )d\gamma \\
&\quad=\mathop{\lim }_{j\rightarrow \infty} \int_{\Omega}  (|\nabla u_{\varepsilon,j}|^2+\delta)^{\frac{p-2}{2}} (\nabla u_{\varepsilon,j} \cdot \nabla\psi  -(x,\nabla u_{\varepsilon,j}) \psi) d\gamma\\
&\quad =-\mathop{\lim }_{j\rightarrow \infty} \int_{\Omega} \left(\mathrm{div} ((|\nabla u_{\varepsilon,j}|^2+\delta)^{\frac{p-2}{2}} \nabla u_{\varepsilon,j}) \right) \psi d\gamma\\
&\quad\leq - \int_{\Omega}\mathop{\lim \inf}_{j\rightarrow \infty} \left(\mathrm{div} ((|\nabla u_{\varepsilon,j}|^2+\delta)^{\frac{p-2}{2}} \nabla u_{\varepsilon,j}) \right) \psi d\gamma\\
&\quad =-\int_\Omega \mathrm{div} \left((|\nabla u_\varepsilon|^2+\delta)^{\frac{p-2}{2}} \nabla u_\varepsilon\right) \psi d\gamma.
\end{align*}
Therefore, we conclude that \eqref{3-12} holds.

Next, we prove
\begin{equation}\label{3-13}
\mathrm{div} \left((|\nabla u_\varepsilon|^2+\delta)^{\frac{p-2}{2}} \nabla u_\varepsilon\right) \geq -c_3,
\end{equation}
a.e. in $\Omega_{r(\varepsilon)}$, where $c_3$ is a positive constant independent of $\delta$.
Indeed, consider a point $x\in \Omega_{r(\varepsilon)}$ where both $\nabla u_\varepsilon(x)$ and $D^2 u_\varepsilon(x)$ exist. By a direct calculation, we have
\begin{equation}\label{3-14}
\mathrm{div} \left((|\nabla u_\varepsilon|^2+\delta)^{\frac{p-2}{2}} \nabla u_\varepsilon\right)=(|\nabla u_\varepsilon|^2+\delta)^{\frac{p-2}{2}}\left( \Delta u_\varepsilon+\frac{p-2}{|\nabla u_\varepsilon|^2+\delta}D^2 u_\varepsilon \nabla u_\varepsilon \cdot \nabla u_\varepsilon \right).
\end{equation}
By Proposition \ref{inf-convo-prop}, we know that
\[
 D^2 u_\varepsilon \geq- \frac{q-1}{\varepsilon}|\nabla u_\varepsilon|^{\frac{q-2}{q-1}}I.
\]
Notice that, if $\nabla u_\varepsilon(x)=0$, we have $D^2 u_\varepsilon(x)\geq 0$, hence from  \eqref{3-14} we get
\begin{equation}\label{3-15}
\mathrm{div} \left((|\nabla u_\varepsilon|^2+\delta)^{\frac{p-2}{2}} \nabla u_\varepsilon\right)\geq 0.
\end{equation}
If $\nabla u_\varepsilon(x)\neq 0$, we have
\begin{align*}
 &(|\nabla u_\varepsilon|^2+\delta)^{\frac{p-2}{2}}\left( \Delta u_\varepsilon+\frac{p-2}{|\nabla u_\varepsilon|^2+\delta}D^2 u_\varepsilon \nabla u_\varepsilon \cdot \nabla u_\varepsilon \right)\\
 &\quad \geq-\frac{(n+p-2)(q-1)}{\varepsilon} |\nabla u_\varepsilon|^{p-2+\frac{q-2}{q-1}}\\
 &\quad \geq-c_3,
\end{align*}
where the last inequality follows from $q>\frac{p}{p-1}$ and the Lipschitz continuity of $u_\varepsilon$. Combining \eqref{3-14} and  \eqref{3-15} gives \eqref{3-13}.

Thus, we can use Fatou's Lemma to conclude that
\begin{align*}
&-\mathop {\lim \inf}_{\delta\rightarrow 0}\int_{\Omega}\mathrm{div} \left((|\nabla u_\varepsilon|^2+\delta)^{\frac{p-2}{2}} \nabla u_\varepsilon\right)  \psi d\gamma+\int_\Omega (x,\nabla u_\varepsilon)|\nabla u_\varepsilon|^{p-2}\psi d\gamma\\
&\quad \leq -\mathop {\lim \inf}_{\delta\rightarrow 0}\int_{\Omega \setminus \{ \nabla u_\varepsilon= 0\}}\mathrm{div} \left((|\nabla u_\varepsilon|^2+\delta)^{\frac{p-2}{2}} \nabla u_\varepsilon\right)  \psi d\gamma+\int_\Omega (x,\nabla u_\varepsilon)|\nabla u_\varepsilon|^{p-2}\psi d\gamma\\
&\quad \leq -\int_{\Omega \setminus \{\nabla u_\varepsilon=0 \}} \mathop {\lim \inf}_{\delta\rightarrow 0} \mathrm{div} \big((|\nabla u_\varepsilon|^2+\delta)^{\frac{p-2}{2}} \nabla u_\varepsilon\big)   \psi d\gamma+\int_\Omega (x,\nabla u_\varepsilon)|\nabla u_\varepsilon|^{p-2}\psi d\gamma\\
&\quad = \int_{\Omega \setminus \{\nabla u_\varepsilon=0 \}} (-\mathrm{div} \left(|\nabla u_\varepsilon|^{p-2} \nabla u_\varepsilon \right) +(x,\nabla u_\varepsilon)|\nabla u_\varepsilon|^{p-2} ) \psi d\gamma\\
&\quad =\int_{\Omega \setminus \{\nabla u_\varepsilon=0 \}} -\Delta_{p,\gamma}u_\varepsilon \cdot \psi d\gamma\\
&\quad \leq \int_{\Omega\setminus \{\nabla u_\varepsilon=0 \} }f_\varepsilon \psi d\gamma,
\end{align*}
where the last inequality follows from \eqref{3-11}. Letting $\delta\rightarrow 0$ in \eqref{3-12}, we have
\[
\int_\Omega |\nabla u_\varepsilon|^{p-2} \nabla u_\varepsilon \cdot \nabla \psi d\gamma \leq \int_{\Omega \setminus \{\nabla u_\varepsilon=0 \}} f_\varepsilon \psi d\gamma\leq \int_{\Omega} f_\varepsilon \psi d\gamma,
\]
where in the last equality we have used  that $f_\varepsilon$ and $\psi$ are non-negative. By Lemma \ref{weak-converges} and Lemma \ref{error-term}, letting $\varepsilon \rightarrow 0$, we have
\[
\int_\Omega |\nabla u|^{p-2} \nabla {u} \cdot \nabla \psi d\gamma - \lambda_{p,\gamma} \int_{\Omega} |u|^{p-2} u \psi d\gamma\leq 0,
\]
for any non-negative $\psi \in C^{\infty}_0(\Omega)$. This completes its proof.
\end{proof}

Similarly, by using inf-convolution in place of sup-convolution, we can deduce the following.
\begin{proposition}\label{remark1}
 Given $p>1$. Suppose that $u:\Omega \rightarrow \R$ is bounded, non-negative, lower semi-continuous. Then, if $u$ is a viscosity super-solution to \eqref{PDE1}, that it is a distributional super-solution of \eqref{PDE1} in Sobolev space $W_{\mathrm{loc}}^{1,p}(\Omega,\gamma)$.
\end{proposition}

\begin{remark}
Combining above two facts, we conclude that if $u\in C(\overline{\Omega})$ is a non-negative viscosity solution of the equation in \eqref{PDE1}, with $p>1$, then it is a local weak solution of \eqref{PDE1}.
\end{remark}

\section{Proofs of the main theorems}\label{Se4}

\subsection{A Brunn-Minkowski inequality}

\begin{proof}[\bf Proof of Theorem \ref{T1}.]
Let $\Omega_0$ and $\Omega_1$ be open bounded Lipschitz domains and let $u_i$ be a nonegative solution of problem \eqref{PDE1}, for $i=0,1$.
For simplicity, set
\begin{equation}\label{4-2}
\lambda_i=\lambda_{p,\gamma}(\Omega_i)\ \ \mathrm{for}\ \ i=0,1,\ \ \lambda_t=(1-t)\lambda_0+t\lambda_1.
\end{equation}
For $x\in \Omega_t$, let
\begin{equation}\label{4-3}
u_t(x)=\sup \left\{ u_0(x_0)^{1-t}u_1(x_1)^t:\ x_i \in \overline{\Omega}_i\ \text{for}\ i=0,1, x=(1-t)x_0+tx_1\right\}.
\end{equation}
For every $\bar{x}\in \overline{\Omega}_t$, there exists $\bar{x}_0\in \overline{\Omega}_0$, $\overline{x}_1\in \overline{\Omega}_1$ where the maximum in the definition \eqref{4-3} is attained, that is
\begin{equation}\label{4-4}
\bar{x}=(1-t)\bar{x}_0+t\bar{x}_1, \ \ u_t(\bar{x})=u(\bar{x}_0)^{1-t}u_1(\bar{x}_1)^t.
\end{equation}
Notice that if $\bar{x}\in \partial \Omega_t$, then $\bar{x}_i\in \partial \Omega_i$ for $i=0,1$. While, if $\bar{x}\in \Omega_t$,  by the boundary condition, we know that  $\bar{x}_i \in \Omega_i$ for $i=0,1$. In the latter case, as $u_0$ and $u_1$ are differentiable at $x_0$ and $x_1$, respectively, a straightforward consequence of the Lagrange multipliers theorem is that
\begin{equation}\label{4-5}
\frac{\nabla u_0(\bar{x}_0)}{u_0(\bar{x}_0)}=\frac{\nabla u_1(\bar{x}_1)}{u_1(\bar{x}_1)}=\theta.
\end{equation}
Next, we prove that $u_t$ is a viscosity sub-solution of the problem
\begin{equation}\label{4-6}
\left\{
\begin{aligned}
& -\Delta_p u_t+(x,\nabla u_t)|\nabla u_t|^{p-2}=\lambda_t|\nabla u_t|^{p-1}u_t, \ \ \ \ \ \  \mathrm {in} \ \ \Omega_t  ,  \\
& \mathop {\lim}_{x\rightarrow \partial \Omega} u_t=0.
\end{aligned}
\right.
\end{equation}
Let us consider two cases: $\theta\neq 0$ and $\theta=0$.

Case 1: $\theta\neq 0$.  In a sufficiently small neighborhood $B$ of $\bar{x}$, we define
\[
\psi(x)=u_0(x-\bar{x}+\bar{x}_0)^{1-t}u_1(x-\bar{x}+\bar{x}_1)^t.
\]
It is obvious that $\psi$ touches $u_t$ from above.
Moreover, we have
\[
\nabla \psi(x)=\psi(x)\left[ (1-t)\frac{\nabla u_0(x-\bar{x}+x_0)}{ u_0(x-\bar{x}+x_0)}+t\frac{\nabla u_1(x-\bar{x}+x_1)}{ u_1(x-\bar{x}+x_1)}\right]
\]
and
\begin{align*}
D^2 \psi(x)&=-\psi(x)\left[ (1-t)\frac{\nabla u_0 \otimes \nabla u_0}{u^2_0}(x-\bar{x}+\bar{x}_0)+t \frac{\nabla u_0 \otimes \nabla u_0}{u^2_0}(x-\bar{x}+\bar{x}_0)\right]\\
&\quad +\frac{\nabla \psi \otimes \nabla \psi}{\psi}(x)+ \psi(x)\left[ (1-t)\frac{D^2 u_0(x-\bar{x}+x_0)}{ u_0(x-\bar{x}+x_0)}+t\frac{D^2 u_1(x-\bar{x}+x_1)}{u_1(x-\bar{x}+x_1)}\right].
\end{align*}
Therefore, we have
\begin{equation}\label{4-7}
D^2 \psi (\bar{x})=\psi(\bar{x}) \left[ (1-t)\frac{D^2 u_0}{u_0}(\bar{x}_0)+t\frac{D^2 u_1}{u_1}(\bar{x}_1)\right],
\end{equation}
and
\begin{equation}\label{4-8}
\nabla \psi(\bar{x})=\psi(\bar{x})\theta\neq 0.
\end{equation}
By \eqref{p-Lapla-form}, \eqref{4-5}, \eqref{4-7} and \eqref{4-8}, we have
\begin{align*}
\Delta_p \psi(\bar{x})&= \psi(\bar{x})^{p-1}\left[ (1-t)\frac{\Delta_p u_0(\bar{x}_0)}{u_0(\bar{x}_0)^{p-1}}+t\frac{\Delta_p u_1(\bar{x}_1)}{u_1(\bar{x}_1)^{p-1}}\right] \\
&= (1-t)\psi(\bar{x})^{p-1}\left[(x,\nabla u_0(\overline{x}_0))|\nabla u_0(\overline{x}_0)|^{p-2}u_0(\bar{x}_0)^{1-p}+\lambda_0 \right]\\
&\quad +t\psi(\bar{x})^{p-1}\left[(x,\nabla u_1(\overline{x}_1))|\nabla u_1(\overline{x}_1)|^{p-2}u_1(\bar{x}_1)^{1-p}+\lambda_1 \right]\\
& =\psi(\bar{x})^{p-1}\langle \theta, (1-t)\bar{x}_0+t\bar{x}_1 \rangle |\theta|^{p-2} +\lambda_t|\psi(\bar{x})|^{p-1}\\
&=\langle \bar{x},\nabla \psi(\bar{x})\rangle | \psi(\bar{x})|^{p-2}+\lambda_t| \psi(\bar{x})|^{p-1}.
\end{align*}
The above fact yields that every test function $\phi\in C^2(\Omega_t)$ touching $u_t$ from above at $\bar{x}$ must also touch $\psi$ from above at $\overline{x}$, which gives
\begin{equation}\label{4-9}
-\Delta_p \phi(\bar{x})+(\bar{x},\nabla \phi(\bar{x})|\nabla \phi(\bar{x})|^{p-2}\leq \lambda_t| \phi(\bar{x})|^{p-2}\phi(\bar{x}).
\end{equation}

Case 2: $\theta= 0$. We know that $\nabla \phi (\bar{x})=0 $, so, if $p>2$ we have
\begin{equation}\nonumber
\Delta_p \phi(\bar{x})=|\nabla\phi(\bar{x})|^{p-2} \Delta \phi(\bar{x})+(p-2)|\nabla \phi|^{p-4}(D^2 \phi(\bar{x}),\nabla \phi)\cdot \nabla \phi=0,
\end{equation}
hence \eqref{4-9} holds. When $1<p<2$, the definition of viscosity solutions avoids points where $\theta=0$.

Thus, we deduce that $u_t$ is a viscosity sub-solution of problem \eqref{4-6}, that is, we proved that $u_t$ satisfies the inequality
\begin{equation}\label{viscsubsol}
-\Delta_p u_t+(x,\nabla u_t)|\nabla u_t|^{p-2}\leq \lambda_t |u_t|^{p-2}u_t,\ \  \mathrm{ in}\ \Omega_t\,\text{ in the viscosity sense}.
\end{equation}

Since $u_t\in C(\overline{\Omega}_t)$ and $u_t$ vanishes on $\partial \Omega_t$, we extend $u_t$ to $\Omega^{*}$ as follows
$$
\tilde{u}_t(x)=\left\{\begin{array}{ll}
u_t(x)\quad&\text{if }x\in\overline\Omega_t\\
0\quad&\text{in } \Omega^*\setminus{\overline\Omega}_t\,,
\end{array}\right.
$$
where $\Omega_{*}\supseteq \overline\Omega_t$ is an open, bounded Lipschitz domain, and we notice that $\tilde{u}_t\in C(\Omega^{*})$.
We claim that $\tilde{u}_t$ satisfies \eqref{viscsubsol} in the viscosity sense in $\Omega^*$.
Indeed, it is clear that we have to take care only of the points on $\partial\Omega_t$. Let $x\in\partial\Omega_t$ and let $\phi$ be a $C^2$ function touching $\tilde{u}_t$ from above at $x$, then it must hold $\phi(x)=0$ and $\nabla\phi(x)=0$ (because $\tilde{u}_t$ is vanishing outside $\Omega_t$ and  $\Omega_t$ is a Lipschitz domain), so $x$ is not considered in the definition of viscosity subsolution if $p\in(1,2)$, while
$-\Delta_p \phi(x)+(x,\nabla \phi(x))|\nabla \phi(x)|^{p-2}+\lambda_t |\phi(x)|^{p-2}\phi(x)=0$ if $p>2$, and the claim is proved.

Now, by Proposition \ref{Viscosity-weak}, $\tilde{u}_t$ is a local sub-solution of \eqref{PDE1} in $\Omega^*$. Then $u_t$ (which coincides with $\tilde u_t$ in $\overline\Omega_t$ and belongs to $W_0^{1,p}(\Omega_t,\gamma)$, as it is continuous in $\overline\Omega_t$ and vanishes on $\partial\Omega_t$) is a weak sub-solution of problem \eqref{PDE1}. So, taking the same $u_t$ as test function in the very definition of weak sub-solution, an integration by parts gives
\[
\int_{\Omega_t} |\nabla {u}_t|^{p} d\gamma \leq \lambda_t \int_{\Omega_t} |{u}_t|^{p} d\gamma,
\]
whence
\begin{equation}\nonumber
\lambda_{p,\gamma}(\Omega_t)\leq \lambda_t.
\end{equation}
This concludes the proof.
\end{proof}

\begin{proof}[\bf Proof of Corollary \ref{cor-min}.]
{   First of all, let us consider the convex hull $\Omega^*$ of $\Omega$ and notice that $\lambda_{p,\gamma}(\Omega^*)\leq\lambda_{p,\gamma}(\Omega)$ by the monotonicity of the eigenvalue with respect to inclusion. Then, } the proof is exactly the same as  \cite[Corollary 1.5]{Andrea-Paolo2024}.
\end{proof}

\subsection{Log-concavity of the first eigenfunction}

\begin{proof}[\bf Proof of Theorem \ref{T2}.]
Let $u$ be a positive solution of \eqref{PDE1} in $\Omega$. For $x\in \Omega$ and $t\in[0,1]$, we define
\begin{equation}
u_t(x)=\sup \left\{ u(x_0)^{1-t}u(x_1)^t:\ x_i \in \overline{\Omega}\ \text{for}\ i=0,1, x=(1-t)x_0+tx_1\right\}.
\end{equation}
The convex envelope $u^{*}$ of $u$ can be obtained as
\[
u^{*}(x)=\sup \left\{ u_t(x):\ t\in [0,1] \right\}\,.
\]
By definition, it is clear that $u^{*}\geq u_t\geq u$. Moreover, $u$ is log-concave in $\Omega$ if and only if $u^{*}=u=u_t$ for every $t\in[0,1]$. Similarly to the proof of Theorem \ref{T1} (in fact, just taking $\Omega_0=\Omega_1=\Omega$ and $u_0=u_1=u$ therein), we can deduce that  $u_t$ satisfies the inequality
\begin{equation}
-\Delta_p u_t+(x,\nabla u_t)|\nabla u_t|^{p-2}\leq \lambda_{p,\gamma} |u_t|^{p-2}u_t\quad  \mathrm{ in}\ \Omega,
\end{equation}
in viscosity and in distributional sense. Furthermore, $u_t\in W^{1,p}_{0}(\Omega,\gamma)$. Thus, we multiply both terms of the last inequality by $-u_t$ and integrate by parts (with respect to $\gamma$) to get
\[
\int_{\Omega} |\nabla {u}_t|^{p} d\gamma \leq \lambda_{p,\gamma}(\Omega)\int_\Omega|u_t|^pd\gamma,
\]
which implies that $u_t$ is an eigenfunction, by the definition of $ \lambda_{p,\gamma}(\Omega)$. As observed in the introduction, we have
$u_t=\alpha u$ for some $\alpha>0$.
On the other hand, by the definition of $u^{*}$, we have
\[
\mathop{\sup}_{\overline{\Omega}} u^{*}=\mathop{\sup}_{\overline{\Omega}} u_t=\mathop{\sup}_{\overline{\Omega}}u.
\]
Thus, we have $\alpha=1$ and $u_t=u$ in $\Omega$,
which implies that $u$ is log-concave.
\end{proof}


\begin{thebibliography}{99}

\bibitem{BCF} M,F. Betta, F. Chiacchio, A. Ferone, \emph{Isoperimetric estimates for the first eigenfunction of a class of
linear elliptic problems}. Z. Angew. Math. Phys. 58(1)  (2007), 37-52.

\bibitem{BS} M. Bianchini P. Salani, Power concavity for solutions of nonlinear elliptic problems in convex domains, in: R. Magnanini, et al. (Eds.), Geometric Properties for Parabolic and Elliptic PDEs, in: Springer INdAM Ser., vol. 2, 2013, pp. 35-48.

\bibitem{Borell1} C. Borell, \emph{Capacitary inequalities of the Brunn-Minkowski type}, Math. Ann. 263 (1983), 179-184.

\bibitem{Borell3} C. Borell, \emph{Greenian potentials and concavity}, Math.
Ann. 272 (1985), 155-160.

\bibitem{Brascamp-Lieb1976}
H.~J. Brascamp\ and\ E.~H. Lieb, {\em On extensions of the Brunn-Minkowski and Pr\'{e}kopa-Leindler theorems, including inequalities for log concave functions, and with an application to the diffusion equation}, J. Functional Analysis {\bf 22} (1976), no.~4, 366--389.

\bibitem{Caffarelli-Friedman} L. A. Caffarelli, A. Friedman, {\it Convexity of solutions to semilinear elliptic equations}, Duke Math. J. 52 (2) (1985), pp. 431--456.


\bibitem{CCLaMP} A. Carbotti, S. Cito, D.A. La Manna, D. Pallara, \emph{Stability of the Gaussian Faber-Krahn inequality} Ann. Mat. Pura Appl. 203 (2024), 2185-2198.

\bibitem{Colesanti} A. Colesanti, {\it Brunn-Minkowski inequalities for variational functionals and related problems}, Adv. Math. 194 (2005), 105-140.

\bibitem{Colesanti-Salani} A. Colesanti, P. Salani, {\it The Brunn-Minkowski inequality for $p$-capacity of convex bodies}, Math. Ann. 327 (2003) 459--479.

\bibitem{Andrea-Cuoghi-Salani2006}
A. Colesanti, P. Cuoghi\ and\ P. Salani, {\em Brunn-Minkowski inequalities for two functionals involving the $p$-Laplace operator}, Appl. Anal. {\bf 85} (2006), no.~1-3, 45--66.

\bibitem{Andrea-Paolo2024}
A. Colesanti, E. Francini, G. Livshyts\ and\ P. Salani, {\em The Brunn-Minkowski inequalities for the first eigenvalue of the Ornstein-Uhlenbeck operator and log-concavity of the relevant eigenfunction}, {   to appear in Anal. PDE}.

\bibitem{Crandall1995}
M.~G. Crandall, Viscosity solutions: a primer, in {\it Viscosity solutions and applications (Montecatini Terme, 1995)}, 1--43, Lecture Notes in Math., 1660, Fond. CIME/CIME Found. Subser, Springer, Berlin.

\bibitem{Ishii-Lions1992}
M.~G. Crandall, H. Ishii\ and\ P.-L. Lions, {\em User's guide to viscosity solutions of second order partial differential equations}, Bull. Amer. Math. Soc. (N.S.) {\bf 27} (1992), no.~1, 1--67.

\bibitem{Crasta-Fragal2020}
G. Crasta\ and\ I. Fragal\`a, {\em The Brunn-Minkowski inequality for the principal eigenvalue of fully nonlinear homogeneous elliptic operators}, Adv. Math. {\bf 359} (2020), 106855, 24 pp.

\bibitem{Feng2021}
F. Du, J. Mao, Q. Wang\ and C. Xia, {\em Estimates for eigenvalues of weighted Laplacian and weighted $p$-Laplacian}, Hiroshima Math. J. {\bf 51} (2021), no.~3, 335--353.

\bibitem{Ehrhard-84} A. Ehrhard, {\it In\'egalit\'es isop\'erim\'etriques et int\'egrales de Dirichlet gaussiennes},
Annales scientifiques E.N.S. 4e s\'erie, 17, n. 2 (1984), p. 317-332.

\bibitem{Evans1992}
L.~C. Evans\ and\ R.~F. Gariepy, {\it Measure theory and fine properties of functions}, revised edition, Textbooks in Mathematics, CRC, Boca Raton, FL, 2015.

\bibitem{Radulescu2024}
Y. Fang, V. R\u{a}dulescu\ and\ C. Zhang, {\em Equivalence of weak and viscosity solutions for the nonhomogeneous double phase equation}, Math. Ann. {\bf 388} (2024), no.~3, 2519--2559.

\bibitem{Franceschi2024}
V. Franceschi, A. Pinamonti, G. Saracco\ and\ G. Stefani, {\em The Cheeger problem in abstract measure spaces}, J. Lond. Math. Soc. (2) {\bf 109} (2024), no.~1, Paper No. e12840, 55 pp.

\bibitem{Gardner2002}
R.~J. Gardner, {\em The Brunn-Minkowski inequality}, Bull. Amer. Math. Soc. (N.S.) {\bf 39} (2002), no.~3, 355--405.

\bibitem{GariepyZiemer} R. Gariepy, W.P. Ziemer, {\em A regularity condition at the boundary for solutions of quasilinear elliptic equations}, Arch. Rational Mech. Anal. {\bf 67} (1977), no. 1, 25-39.

\bibitem{GT} D. Gilbarg and N. Trudinger, {\em Elliptic partial differential equations of second order}, reprint of the 1998 edition. Classics in Mathematics. Springer-Verlag, Berlin, 2001. xiv+517 pp.


\bibitem{Heinonen2006}
J. Heinonen, T. Kilpel\"ainen and O. Martio, {\it Nonlinear potential theory of degenerate elliptic equations}, Dover, Mineola, NY, 2006.

\bibitem{Ishii} H. Ishii, {\em On the equivalence of two notions of weak solutions, viscosity solutions and distribution solutions}, Funkcial. Ekvac., {\bf 38} (1995), no. 1, 101-120.

\bibitem{Versa-Julin2012} V. Julin\ and\ P. Juutinen, {\em A new proof for the equivalence of weak and viscosity solutions for the $p$-Laplace equation}, Comm. Partial Differential Equations {\bf 37} (2012), no.~5, 934--946.

\bibitem{Juutinen2001}
P. Juutinen, P. Lindqvist\ and\ J.~J. Manfredi, {\em On the equivalence of viscosity solutions and weak solutions for a quasi-linear equation}, SIAM J. Math. Anal. {\bf 33} (2001), no.~3, 699--717.

\bibitem{Katzourakis} N. Katzourakis, {\em An Introduction to Viscosity Solutions for Fully Nonlinear PDE with Applications to Calculus of Variations in $L^\infty$}, Springer, Berlin (2015).

\bibitem{Kennington} A. U. Kennington, {\it Power concavity and boundary value problems}, Indiana Univ. Math. J., 34 (3) (1985), 687--703.

\bibitem{Koike} S. Koike, {\em A Beginner's Guide to theTheory of Viscosity Solutions}, vol13 MJS Memoirs of the Mathematical Society of Japan, Tokyo (2012).

\bibitem{Korevaar1} N. Korevaar, {\it Convex solutions to nonlinear elliptic and parabolic boundary value problems}, Indiana Univ. Math. J. 32 (1983), pp. 603--614.

\bibitem{Korevaar-Lewis} N. Korevaar, J. Lewis, {\it Convex solutions to certain elliptic equations have constant rank Hessians}, Arch. Rat. Mech. Anal. 97 (1987), 19--32.

\bibitem{Ladizeskaya-Ural'ceva-1968} O. A. Ladyzhenskayaand N. N. Ural'tseva,
{\em Linear and quasilinear elliptic equations}, translated from the Russian by Scripta Technica, Inc. Translation editor: Leon Ehrenpreis. Academic Press, New York-London, 1968. xviii+495 pp.

\bibitem{Lieberman} G. Lieberman, {\em Boundary regularity for solutions of degenerate elliptic equations}, Nonlinear Anal. {\bf 12} (1988), no.~11, 1203--1219.

\bibitem{Lindqvist1990} P. Lindqvist, {\em On the equation ${\rm div}\,(|\nabla u|^{p-2}\nabla u)+\lambda|u|^{p-2}u=0$}, Proc. Amer. Math. Soc. {\bf 109} (1990), no.~1, 157--164.

\bibitem{Lions} P.-L. Lions, {\em Optimal control of diffusion processes and Hamilton-Jacobi-Bellman equations. II. Viscosity solutions and uniqueness}, Comm. Partial Differential Equations {\bf 8} (1983), no. 11, 1229-1276.

\bibitem{LiuMaXu} P. Liu, X.-N. Ma, L. Xu, {\em A Brunn-Minkowski inequality for the Hessian eigenvalue in three-dimensional convex domain}, Adv. Math. {\bf 225} (2010), 1616-1633.

\bibitem{MO} M. Medina, P. Ochoa, {\em On viscosity and weak solutions for non-homogeneous $p$-Laplace equations}, Adv. Nonlinear Anal. 8 (2019), no. 1, 468-481.


\bibitem{Lei2} L. Qin, {\em Log-concavity of eigenfunction and Brunn-Minkowski inequality of eigenvalue for weighted $p$-Laplace operator}, arXiv:2411.16377.

\bibitem{Radulescu2019}
N.~S. Papageorgiou, V. R$\mathrm{\check{a}}$dulescu\ and\ D. Repov\v{s}, {\it Nonlinear analysis---theory and methods}, Springer Monographs in Mathematics, Springer, Cham, 2019.

\bibitem{Sakaguchi1987} S. Sakaguchi, {\em Concavity properties of solutions to some degenerate quasilinear elliptic Dirichlet problems}, Ann. Scuola Norm. Sup. Pisa Cl. Sci. (4) {\bf 14} (1987), no.~3, 403--421 (1988).

\bibitem{Salani2005} P. Salani, {\em A Brunn-Minkowski inequality for the Monge-Amp\`ere eigenvalue}, Adv. Math. {\bf 194} (2005), 67--86.

\bibitem{Salani2012}
P. Salani, {\em Convexity of solutions and Brunn-Minkowski inequalities for Hessian equations in $\R^3$}, Adv. Math. {\bf 229} (2012), 1924-1948

\bibitem{S} P. Salani, {\em Combination and mean width rearrangements of solutions to elliptic equations in convex sets}, Ann. I. H. Poincar\'e C Anal. Non Lin\'eaire {\bf 32} (2015), 763-783.

\bibitem{Siltakoski2018}
J. Siltakoski, {\em Equivalence of viscosity and weak solutions for the normalized $p(x)$-Laplacian}, Calc. Var. Partial Differential Equations {\bf 57} (2018), no.~4, Paper No. 95, 20 pp.

\bibitem{Tero1994} T. Kilpel\"{a}inen, {\em Weighted Sobolev spaces and capacity}, Ann. Acad. Sci. Fenn. Ser. A I Math. {\bf 19} (1994), no.~1, 95--113.

\bibitem{Tolksdorf} P. Tolksdorf, {\em Regularity for a more general class of quasilinear elliptic equations},
J. Differential Equations 51 (1984), no. 1, 126--150.


\end{thebibliography}
\end{document}